\pdfoutput=1
\documentclass[11pt]{amsart}
\usepackage{amsmath}
\usepackage{amssymb}
\usepackage{mathtools}

\usepackage{enumerate}
\usepackage{slashed}
\usepackage{caption}
\usepackage{subcaption}
\usepackage{newlfont}
\usepackage{amsrefs}
\usepackage{comment}
\usepackage[normalem]{ulem}
\usepackage{accents}
\usepackage{leftidx}

\usepackage{harpoon}
\usepackage[pdftex]{graphicx}
\usepackage{esint}
\usepackage{relsize}
\usepackage[latin1]{inputenc}
\usepackage{xcolor}

\definecolor{brass}{rgb}{0.71, 0.65, 0.36}

\usepackage{tikz-cd}
\theoremstyle{plain}
\newtheorem{theorem}{Theorem}[section]
\newtheorem{lemma}[theorem]{Lemma}

\theoremstyle{definition}
\newtheorem{remark}[theorem]{Remark}
\newtheorem{definition}[theorem]{Definition}

\numberwithin{equation}{section}

\def\Hr{\mathbb{H}}
\def\Rr{\mathbb{R}}

\def\sup{\operatorname{sup}}
\def\inf{\operatorname{inf}}

\def\sup{\operatorname{sup}}

\theoremstyle{plain}

\numberwithin{equation}{section}

\allowdisplaybreaks

\begin{document}

\title[Stability of the Rigidity of PMT for Static QL Energy]{Stability of Positive Mass Theorem for Static Quasi-Local Energy of Compact (Locally) Hyperbolic Graphical Manifolds}

\author[Alaee]{Aghil Alaee}
\address{
Department of Mathematics, Clark University, Worcester, MA 01610, USA}
\email{aalaeekhangha@clarku.edu}

\author[Liu]{Jiusen Liu}
\address{Department of Mathematics, Universit\"at T\"ubingen,  72076 T\"{u}bingen, Germany.}
\email{jiusen.liu@student.uni-tuebingen.de }


\thanks{A. Alaee acknowledges the support of NSF Grant DMS-2316965.}

\begin{abstract}In this paper, we consider compact graphical manifolds with boundary over (locally) hyperbolic static space. We prove the stability of the positive mass theorem with respect to the Federer--Fleming flat distance for the static quasi-local Brown-York energy of the outer boundary of compact (locally) hyperbolic graphical manifolds.
\end{abstract}

\maketitle

\section{Introduction}
\label{sec1} \setcounter{equation}{0}
\setcounter{section}{1}
Is there a universal definition for quasi-local energy of a finite spatial region of spacetimes? This is an open problem in mathematical general relativity, which was raised initially by Penrose \cite{PenroseR}. Similar to the definition of the total energy of asymptotically flat spacetime in general relativity, the Hamilton-Jacobi methods are used to define quasi-local energies \cite{AlaeeKhuriYau,BrownYork,HH,Tsang,WangYau}; see also survey paper \cite{Szabados}. A primary property of a quasi-local energy, inspired by the total energy \cite{SchoenYau79,Witten}, is the positivity and rigidity statement which is called the positive mass theorem. The positive mass theorem for quasi-local energies, derived from the Hamilton-Jacobi methods, is a reference dependent statement. If the reference is the Minkowski spacetime, then the positivity property claims that the quasi-local energy of a finite region of spacetime, satisfying the dominant energy condition, should be non-negative and the rigidity holds if and only if the domain arises from the Minkowski spacetime. Similarly, if the reference has a non-zero total energy, e.g., Schwarzschild and dS/AdS Schwarzschild, it is conjectured that the mass of reference should appear in the inequality of the positive mass theorem. See \cite{AlaeeKhuriKunduri,AlaeeKhuriYau3,AlaeeKhuriYau4} for recent results.

A relevant question is the stability of the rigidity statement of the positive mass theorem: If the quasi-local energy is close to the rigidity case, is the region close to a region of the reference spacetime in some topologies? The global version of this question, concerning the total energy of asymptotically flat and hyperbolic Riemannian manifolds, has been studied extensively in recent years. Some of these studies include \cite{Allen1, BKS19, DahlSakovichGicquaud, HuangLee, HLS, LeeSormani, S-S}, with additional insights provided in the review paper \cite{Sormanireview} and the references therein. However, for quasi-local energies, this question has only been studied in the context of the rigidity of the positive mass theorem for the Brown-York quasi-local energy \cite{ShiTam02}, by the first author, Cabrera~Pacheco, and McCormick \cite{AlaeepachecoMccormick}. It is desirable to extend this investigation to other quasi-local energies in different settings.

A natural generalization of the Brown-York quasi-local energy is to change the reference spacetime from the Euclidean space arising from Minkowski spacetime to a static manifold arising as a constant-time slice of a static spacetime. Recall that static spacetime is a solution of the Einstein equations with a hypersurface orthogonal timelike Killing vector field. A static manifold is a constant-time slice of static spacetime, and it is a Riemannian manifold $(\mathbb{M},b)$ with a positive potential function $V$ that is a solution of Einstein static equation
\begin{equation}\label{ofnoinohiw}
(\Delta_{b} V) b-\text{Hess}_b V+V\text{Ric}(b)=0,\qquad \Delta_bV+\Lambda V=0\,,
\end{equation}where $\Delta_{b}$, $\text{Hess}_b$, and $\text{Ric}(b)$ are Laplacian, hessian, and Ricci curvature of the metric $b$ and $\Lambda$ is the cosmological constant.

The static Brown-York energy of a closed surface $\Sigma$, which arises as a smooth boundary $\Sigma=\partial\Omega$ of a compact Riemannian manifold $(\Omega^n,g)$,  with respect to static Riemannian manifold $(\mathbb{M},b)$ is 
\begin{equation}
m_{BY}^S(\Sigma)=\frac{1}{(n-1)\mathcal{A}_{n-1}}\int_{\Sigma}V\left(H_0-H\right) dA,
\end{equation}
where $H$ is the mean curvature of $\Sigma$ with respect to metric $g$, $H_0$ is the mean curvature of the isometric embedding surface $\Sigma\hookrightarrow (\mathbb{M},b)$, $dA$ is the volume element of $\Sigma$, and $\mathcal{A}_{n-1}$ is the volume of space-form metric of constant sectional curvature on $\Sigma$. In this paper, the static reference manifold for the static Brown-York energy is the Kottler space $(\mathbb{M}_\epsilon,b_\epsilon=V_\epsilon^2dr^2+r^2h_\epsilon)$ for $\epsilon=0,\pm1$, where $V_\epsilon=\sqrt{r^2+\epsilon}$ and $h_\epsilon$ is space form metric with constant sectional curvature $\epsilon$ on topological $\Sigma$,  which is obtained from a constant time slice of Kottler spacetime. The Kottler space has constant sectional curvature $-1$.  In three dimensions, assuming the Gauss curvature of topological sphere $\Sigma$ to be larger than minus one, there exists a unique isometric embedding in the hyperbolic space $(\mathbb{M}^n_1=\mathbb{H}^n,b_1=\left(r^2+1\right)^{-1}dr^2+r^2 g_{S^{n-1}})$ \cite{Pogorelov}. Therefore, this quasi-local energy is a quantity that can be used for mathematical results and physical applications. However, there is no such isometric embedding result in higher dimensions.

Consider a compact three-dimensional Riemannian manifold $(\Omega,g)$ with scalar curvature not less than minus six and smooth boundary $\partial\Omega=\Sigma$ with positive mean curvature and Gauss curvature larger than minus one, then Wang and Yau proved the positivity of the static Brown-York quasi-local energy of $\Sigma$ with respect to the hyperbolic space \cite{WangYau2007}. Moreover, Shi and Tam showed that the equality holds if and only if $(\Omega,g)$ is diffeomorphic to a compact domain in hyperbolic space \cite{ShiTam1}. In \cite[Theorem B' (iii),
Section 3.2.4]{Gromov}, Gromov showed that any metric on the 2-torus can be isometrically embedded in a scaled
Kottler manifold $(\mathbb{M}^n_0,b_0=r^{-2}dr^2+r^2 g_{T^{n-1}})$ with sufficiently large negative curvature. Recently, the first author, Hung, and Khuri proved a positive mass theorem for the static Brown-York mass of 2-torus with respect to Kottler space with $\epsilon=0$ \cite{AlaeeHungKhuri} based on a global result for asymptotically hyperbolic manifolds with toroidal infinity \cite
{AlaeeKhuriHung}. These results promote further investigation of the static Brown-York energy and, in particular, the stability of the rigidity of these positive mass theorems. 

In this paper, we study $(n-1)$-dimensional closed hypersurface $\Sigma$ as the outer boundary of compact graphical manifolds over static reference  $(\mathbb{M}_\epsilon,b_\epsilon)$ in the $(n+1)$-dimensional (universal cover) static manifold $(\mathbb{M}^{n+1}_\epsilon=\mathbb{R}\times\mathbb{M}^n_\epsilon,\bar{b}_\epsilon=V^2_\epsilon dt^2+b_\epsilon)$. To state our main theorem, we define the class of manifolds that we will investigate here.

\begin{definition}\label{def1}
    Assume $U_o\subset U$ are two bounded open sets in $\Hr^n$, with $U$ and $U_o$ are connected. A \emph{compact hyperbolic graphical manifold} $(\Omega^n,g)$ refer to a smooth submanifold of $(\Hr^{n+1},\bar{b}_1)$ with upward pointing mean curvature vector and smooth boundary $\partial \Omega=\Sigma\cup\Sigma_o$ that can be realized as a graph of a smooth function $f: \overline{U}\backslash U_o\to \Rr $ such that
    \begin{enumerate}
    \item $f$ has a finite range,
    \item every regular level set $\Sigma_h=\{(x,f(x))\in\mathbb{H}^{n+1}: x\in\overline{U}\backslash U_o,\,f(x)=h\}$ is star-shaped and outer-minimizing in hyperplane $\mathbb{H}^n$, and
    \item outer boundary $\Sigma=\{(x,f(x))\in \Hr^{n+1}:x\in\partial U\}$ is level set and inner boundary $\Sigma_o=\{(x,f(x))\in \Hr^{n+1}:x\in\partial U_o\}$ is minimal submanifold which $f$ is locally constant.
    \end{enumerate}
\end{definition} If, in addition, the scalar curvature $R(g)\geq -n(n-1)$, then every regular level set of compact hyperbolic graphical manifold $\Omega$ is strictly mean convex in $\mathbb{H}^n$, see \cite[Lemma 2.1]{CabreraPacheco19}. Next, we define the locally hyperbolic graphical manifolds as follows.
\begin{definition}\label{def2}Let $\epsilon=0,-1$. 
    Assume $U_o\subset U$ open sets in $\mathbb{M}_\epsilon^n$ with $U$ and $U_o$ are connected and $U_o$ includes cuspical end for $\epsilon=0$ and inner boundary for $\epsilon=-1$ in $\mathbb{M}_\epsilon^n$, such that $\overline{U}\backslash U_o$ is bounded. A \emph{compact locally hyperbolic graphical manifold} $(\Omega^n,g)$ refer to a smooth submanifold of $(\mathbb{M}_\epsilon^{n+1},\bar{b}_\epsilon)$ with smooth boundary $\partial \Omega=\Sigma\cup\Sigma_o$ that can be realized as a graph of a smooth function $f: \overline{U}\backslash U_o\to \Rr $ such that
    \begin{enumerate}
    \item $f$ has a finite range,
    \item every regular level set $\Sigma_h=\{(x,f(x))\in\mathbb{M}_\epsilon^{n+1}: x\in\overline{U}\backslash U_o,\,f(x)=h\}$ is star-shaped, mean convex, and outer-minimizing in hyperspace $\mathbb{M}_\epsilon^{n}$,
    \item outer boundary $\Sigma=\{(x,f(x))\in \mathbb{M}_\epsilon^{n+1}:x\in\partial U\}$  is level set and inner boundary $\Sigma_o=\{(x,f(x))\in \Hr^{n+1}:x\in\partial U_o\}$ is  minimal submanifold which $f$ is locally constant.
    \end{enumerate}
\end{definition}
The main theorem is as follows.
\begin{theorem}\label{main result2}
Let $(\Omega_i,g_i)$ be a sequence of compact (locally) hyperbolic graphical manifolds with outer boundary $\Sigma_i$, as in Definition \ref{def1}-\ref{def2}, and the scalar curvature $R(g_i)\geq -n(n-1)$.  After an appropriate normalization, if $m_{BY}^S(\Sigma_i)\to 0$, then the sequence $\Omega_i$ converges to $\{0\}\times \left(\overline{U}\backslash U_o\right)$ in the sense of currents. Moreover, $\operatorname{vol}(\Omega_i)\to \operatorname{vol}(\overline{U}\backslash U_o)$.
\end{theorem}
The proof of this theorem is based on \cite{AlaeepachecoMccormick} and \cite{HuangLee} that is tailored to our setting. This paper is outlined as follows. In Section 2, we study the static Brown-York energy of compact hyperbolic graphs. In Section 3, we establish a volume estimate which leads to convergence of volume in our theorem. In Section 4, we prove an estimate for flat distance in terms of quasi-local energy and prove the main theorem.

\section{Static Brown-York Quasi-local energy}
In this section, we study the static Brown-York quasi-local energy of $\Sigma$ with respect to $(\mathbb{M}^n_\epsilon,b_\epsilon)$, where $\Sigma$ is the outer boundary of a compact (locally) hyperbolic graphical manifold $(\Omega,g)$. Let $\nabla$ and $D$ denote the Levi-Civita connection with respect to $(\Omega,g)$ and static reference $(\mathbb{M}^n_\epsilon,b_\epsilon)$, respectively. Moreover, let $\mathring{H}_{\Sigma_h}$ and ${H}_{\Sigma_h}$ represent the mean curvature of level set $\Sigma_h$ in static reference $(\mathbb{M}^n_\epsilon,b_\epsilon)$ and $\Omega$, respectively. We find a lower bound for the static Brown-York mass of $\Sigma$.

\begin{lemma}\label{lemma2.1} Let $(\Omega,g)$ be a compact (locally) hyperbolic graphical manifold, as in Definition \ref{def1} (resp., Definition \ref{def2}), with scalar curvature $R(g)\geq -n(n-1)$ with outer boundary $\Sigma$. Then 
\begin{equation}\label{lower.1}
m_{BY}^{S}(\Sigma)\geq \frac{1}{2(n-1)\mathcal{A}_{n-1}}\int_{\Sigma}V_\epsilon\frac{V^2_\epsilon|Df|^2}{1+V^2_\epsilon|Df|^2}\mathring{H}_{\Sigma}\,dA\,,
\end{equation}
where $dA$ is the volume element of the level set $\Sigma$. If $\Sigma_o$ is a level set and not a minimal surface, then the equality holds if and only if $(\Omega,g)$ is isometric to a compact domain in  $\mathbb{M}^n_\epsilon$ for $\epsilon=0,\pm 1$.
\end{lemma}
\begin{proof}Let $V=V_{\epsilon}$, for any surface $\Sigma$ that can be isometrically embedded as a mean convex surface in any static reference, we can find a lower bound for the static Brown-York quasi-local energy using an elementary computation.
\begin{equation}\label{eq1}
m_{BY}^S(\Sigma)\geq \frac{1}{2(n-1)\mathcal{A}_{n-1}}\int_{\Sigma}{VH_o\left(1-\frac{H^2}{H_o^2} \right)}\,dA\,,
\end{equation}
where we used the following quadratic identity.
\begin{equation}\label{eq3}
0\leq \int_{\Sigma} \frac{V}{H_o}\left(H_o-H\right)^2\,dA=\int_{\Sigma}\left(2V\left(H_o-H\right)-V\left(H_o-\frac{H^2}{H_o}\right)\right)\,dA.  \end{equation}

A straightforward gives the relation of the mean curvature of any level set $\Sigma_h$ in static refrence and $\Omega$.
\begin{equation}    \label{eq6}  
H_{\Sigma_h}=\frac{1}{\sqrt{1+V^2|D f|^2}}\mathring{H}_{\Sigma_h}\,.
\end{equation}
We can rewrite the lower bound of \eqref{lower.1} this relation of the mean curvature. 
\begin{equation}\label{eq2}
\begin{split}
\int_{\Sigma_h}V\frac{V^2|Df|^2}{1+V^2|D f|^2}\mathring{H}_{\Sigma_h}\,dA_h&=\int_{\Sigma_h}V\left(1-\frac{1}{1+V^2|D f|^2} \right) \mathring{H}_{\Sigma_h}\,dA_h\\
&=\int_{\Sigma_h}V\left(1-\frac{\mathring{H}_{\Sigma_h}^2}{\mathring{H}_{\Sigma_h}^2\left(1+V^2|D f|\right)^2} \right)\mathring{H}_{\Sigma_h}\,dA_h\\      &=\int_{\Sigma_h}V\mathring{H}_{\Sigma_h}\left(1-\frac{H_{\Sigma_h}^2}{\mathring{H}_{\Sigma_h}^2}\right)\,dA_h\,,
\end{split}
\end{equation}
where $dA_h$ is the volume element of level set $\Sigma_h$. Now if $\Sigma$ is a level set, then combining  \eqref{eq1} and \eqref{eq2} for $\Sigma_h=\Sigma$, we obtain the inequality \eqref{lower.1}. 

If equality holds, then by \eqref{eq3} we have $H_{\Sigma}=\mathring{H}_{\Sigma}$ and $m_{BY}^S(\Sigma)=0$. Combining with mean convexity of $\Sigma$ and \eqref{lower.1}, we have $|Df|=0$ on $\Sigma$. On other hand,  by \cite[Lemma 3.2]{DahlSakovichGicquaud} for any regular values $h_1<h_2$, we have 
\begin{equation}\label{eq4}
\begin{split}
\int_{\mathcal{U}_{h_2}\backslash\,\mathcal{U}_{h_1}}\frac{V\left(R(g)+n\left(n-1\right)\right)}{\sqrt{1+V^2|Df|^2}}\,dx&= \int_{\Sigma_{h_2}}V\left(\frac{V^2|D f|^2}{1+V^2|D f|^2} \right)\mathring{H}_{\Sigma_{h_2}}\,dA_{h_2}\\&-
\int_{\Sigma_{h_1}}V\left(\frac{V^2|D f|^2}{1+V^2|Df|^2} \right)\mathring{H}_{\Sigma_{h_1}}\,dA_{h_1}\,,
\end{split}
\end{equation}
where $dx$ is the volume form with respect to $g$, $\mathcal{U}_{h_i}=\{x\in \overline{U}\backslash U_o: f(x)<h_i\}$ and $\Sigma_{h_i}$ is the reduced boundary $\partial^*\Omega_{h_i}$, which by the Sard's theorem it is precisely the level set $\{f=h_i\}$. Therefore, using $R(g)\geq-n(n-1)$ we have
\begin{equation}
0=\int_{\Sigma}V\frac{V^2|Df|^2}{1+V^2|Df|^2}\mathring{H}_{\Sigma}\,dA\geq \int_{\Sigma_h}V\frac{V^2|Df|^2}{1+V^2|Df|^2}\mathring{H}_{\Sigma_h}\,dA_h\,.
\end{equation} Since every level set is mean convex, we have $|Df|=0$ for all $\Sigma_h$. Therefore, $(\Omega,g)$ is isometric to a compact domain in the static reference.
\end{proof}
Next, using the Minkowski inequality \cite[Theorem 1.1, Theorem 2.3]{delimaGirao}, we prove a quasi-local Riemannian Penrose inequality for compact (locally) hyperbolic graphical manifolds, which is useful for our stability result. 
\begin{lemma}\label{lemma2.2}
Let $(\Omega,g)$ be a compact (locally) hyperbolic graphical manifold, as in Definition \ref{def1} (resp., Definition \ref{def2}), with scalar curvature $R(g)\geq -n(n-1)$. Then we have the following quasi-local Riemannian Penrose inequality
\begin{equation}\label{penrose.ineq}
{m}_{BY}^S(\Sigma)\geq \frac{c_\epsilon}{2}\left(\frac{|\Sigma_o|}{\mathcal{A}_{n-1}}\right)^{n_\epsilon}\,,
\end{equation}
where
\begin{equation}
c_\epsilon=\begin{cases}
1 &\epsilon =1\\
r_0^{\frac{n+1}{2}}& \epsilon =0\\
\sqrt{r_0^2-1} r_0^{\frac{n-3}{2}}& \epsilon =-1,
\end{cases},\qquad\qquad n_\epsilon=\begin{cases}
\frac{n-2}{n-1} &\epsilon =1\\
\frac{1}{2}& \epsilon =0,-1
    \end{cases}, 
\end{equation}where $r_0$ is the radius of the largest $r$-constant surface in $U_o$.
\end{lemma}
\begin{proof}
Combining positivity of $V=V_\epsilon$, equation \eqref{eq4}, $|Df|\to\infty$ on $\Sigma_o$, and the Lemma \ref{lemma2.1}, we have
\begin{equation}\label{eq10}
\begin{split}
{m}_{BY}^S(\Sigma)&\geq\frac{1}{2(n-1)\mathcal{A}_{n-1}}\int_{\Sigma_o}V_\epsilon\mathring{H}_{\Sigma_o}\,dA\,.
\end{split}
\end{equation}By the Minkowski inequality \cite[Theorem 1.1, Theorem 2.3]{delimaGirao}, we have 

\begin{equation}\label{min1}
\begin{split}
\int_{\Sigma_o }V\mathring{H}_{\Sigma_o}dA&\geq (n-1)\mathcal{A}_{n-1}\left(\left(\frac{|\Sigma_o|}{\mathcal{A}_{n-1}}\right)^{\frac{n}{n-1}}+\epsilon\left(\frac{|\Sigma_o|}{\mathcal{A}_{n-1}}\right)^{\frac{n-2}{n-1}}\right)\,.
    \end{split}
\end{equation} When $\epsilon=1$, a lower bound is 
\begin{equation}
\int_{\Sigma_o }V\mathring{H}_{\Sigma_o}dA\geq (n-1)\mathcal{A}_{n-1}\left(\frac{|\Sigma_o|}{\mathcal{A}_{n-1}}\right)^{\frac{n-2}{n-1}}\,.
\end{equation}When $\epsilon=0,-1$, assuming $U_o$ includes a constant $r$-surface with largest radius $r_0$, we have $|\Sigma_0|\geq r_0^{n-1}\mathcal{A}_{n-1}$. Therefore, for $\epsilon=0$, the lower bound of \eqref{min1} is 
\begin{equation}
\int_{\Sigma_o }V\mathring{H}_{\Sigma_o}dA\geq (n-1)r_0^{\frac{n+1}{2}}\mathcal{A}_{n-1}\left(\frac{|\Sigma_o|}{\mathcal{A}_{n-1}}\right)^{\frac{1}{2}}\,.
\end{equation}Furthermore, for $\epsilon=-1$, the lower bound of \eqref{min1} is 
\begin{equation}
 \begin{split}
\int_{\Sigma_o }V\mathring{H}_{\Sigma_o}dA&\geq (n-1)\mathcal{A}_{n-1}\left(\frac{|\Sigma_o|}{\mathcal{A}_{n-1}}\right)^{\frac{n-2}{n-1}}\left(\left(\frac{|\Sigma_o|}{\mathcal{A}_{n-1}}\right)^{\frac{2}{n-1}}-1\right)\\
&\geq (n-1)\sqrt{r_0^2-1}r_0^{\frac{n-3}{2}}\mathcal{A}_{n-1}\left(\frac{|\Sigma_o|}{\mathcal{A}_{n-1}}\right)^{\frac{1}{2}}\,.
\end{split}
\end{equation} Therefore, we have 
\begin{equation}\label{min2}
\begin{split}
\int_{\Sigma_o }V\mathring{H}_{\Sigma_o}dA&\geq (n-1)\mathcal{A}_{n-1}c_{\epsilon} \left(\frac{|\Sigma_o|}{\mathcal{A}_{n-1}}\right)^{n_{\epsilon}}\,.
    \end{split}
\end{equation}This complete the proof.
\end{proof}

\begin{remark} The Minkowski inequality used in Lemma \ref{lemma2.2} differs from the one employed in the proof of the stability of the positive mass theorem for asymptotically hyperbolic graphs \cite[Equation (4.4)]{CabreraPacheco19}. Specifically, the Minkowski inequality \cite[Equation (4.4)]{CabreraPacheco19} is missing a $\kappa$ on the right-hand side; see \eqref{Minfinal}. This discrepancy implies that the estimate in \cite[Lemma 4.8]{CabreraPacheco19} is incorrect. Consequently, the main stability result \cite[Theorem 1.1]{CabreraPacheco19} does not hold when relying on this estimate. However, a result similar to \cite[Theorem 1.1]{CabreraPacheco19} should hold if one replaced \cite[Equation (4.4)]{CabreraPacheco19} with Minkowski inequality \cite[Theorem 1.1, Theorem 2.3]{delimaGirao}.

Following \cite{MontielRos}, we provide a proof of Minkowski inequality \cite[Equation (4.4)]{CabreraPacheco19} in hyperbolic space $\mathbb{H}^n_{-\kappa^2}$, for $\kappa>0$. Consider the hyperbloidal model of hyperbolic space $(\Hr^n_{-\kappa^2},b_\kappa=d{r}^2+\kappa^{-2}\sinh^2\kappa{r}g_{S^{n-1}})$, with constant curvature $-\kappa^2$, which is a graph in the Minkowski space $(\mathbb{R}^{n,1},\eta=\langle \cdot,\cdot\rangle)$. Consider an immersion $\psi:{\Sigma}\to \Hr^n_{-\kappa^2}\subset\mathbb{R}^{n,1}$ with $|\psi|^2_\eta=-\frac{1}{\kappa^2}$ and $\psi^0=\langle\psi,\partial_t\rangle\geq \frac{1}{\kappa}$. 

Let $\mathbf{\mathring{H}}_{{\Sigma}}=-\kappa^2(n-1)\psi+\mathring{H}_{{\Sigma}} N$ be the mean curvature vector of ${\Sigma}$ in $\mathbb{R}^{n,1}$, where $\{\kappa\psi,N\}$ is an orthonormal frame for the normal bundle of ${\Sigma}$ in $\mathbb{R}^{n,1}$ such that $\kappa\psi$ is a unit timelike normal vector and $N$ is unit spacelike normal. Let $\Delta_{{\Sigma}}$ denote the Laplacian of the induced metric on ${\Sigma}$, then  $\Delta_{{\Sigma}}\psi=\mathbf{\mathring{H}}_{{\Sigma}}$ and for any constant $a\in\mathbb{R}^{n,1}$ we obtain
\begin{equation}\label{Min1}
0=\int_{{\Sigma}}\Delta_{{\Sigma}}\langle\psi,a\rangle \,d{A}=\int_{{\Sigma}}\left(-\kappa^2\left(n-1\right)\langle\psi,a\rangle+\mathring{H}_{{\Sigma}}\langle N, a\rangle\right)\, d{A},
\end{equation}
where $d{A}$ is volume element of ${\Sigma}$. If $a=(1,0,\cdots,0)\in\mathbb{R}^{n,1}$, then we have the following relation
\begin{equation}
\int_{{\Sigma}}\left(\kappa^2\left(n-1\right)\psi^0+\mathring{H}_{{\Sigma}}N^0\right)d{A}=0.
\end{equation}

If ${\Sigma}$ is a mean convex level set of a function ${f}:\mathbb{H}^n_{-\kappa^2}\to\mathbb{R}$, the unit normal of ${\Sigma}$ in $\mathbb{R}^{n,1}$ is $N=\frac{(-1, V_\kappa{D} {f} )}{\sqrt{1+V_\kappa^2|{D} \tilde{f}|^2}}$, where $V_\kappa=\cosh\kappa{r}$ and ${D}$ is the Levi-Civita connection with respect to $b_\kappa$. 
Together with $\psi^0\geq \frac{1}{\kappa}$ and $-N^0=\frac{1}{\sqrt{1+V^2_\kappa|{D}{f}|^2}}<1$ , we have 
\begin{equation}\label{Minfinal}
\begin{split}
\int_{{\Sigma}}\mathring{H}_{{\Sigma}}\,d{A}&> -\int_{{\Sigma}}\mathring{H}_{\Sigma}N^0\,d{A}\\
   &\geq \left(n-1\right)\int_{{\Sigma}}\kappa^2\psi^0\,d{A}\\
   &\geq \left(n-1\right)\kappa|{\Sigma}|\,.
\end{split}
\end{equation}
This is a Minkowski inequality in $\mathbb{H}^n_{-\kappa^2}$ that is not sharp. 
\end{remark}

\section{A Volume Estimate}
In this section, we obtain an estimate for the volume of compact (locally) hyperbolic graphical manifold $(\Omega,g)$ in terms of static potential $V=V_\epsilon$, the volume of $\overline{U}\backslash U_o$, geometry of $\Omega$, $\epsilon$, $c_{\epsilon}$, $\xi\geq 1$, the static Brown-York quasi-local energy $\mathbf{m}:={m}_{BY}^S(\Sigma)$ of the outer boundary of $\Omega$, and $\min f$. Let $\mathcal{V}(h)$ denote the volume of regular level set $\Sigma_h$ in $\Omega$. We define a constant 
\begin{equation}
h_o:= \sup\left\{h:\mathcal{V}(h)\leq 2(1+\xi)^{\frac{1}{n_\epsilon}} \mathcal{A}_{n-1}\left(2c^{-1}_\epsilon \mathbf{m}\right)^{\frac{1}{n_\epsilon}}\right\},
\end{equation}
where $\xi\geq 1$, and $n_\epsilon$ and $c_\epsilon$ are defined in Lemma \ref{lemma2.2}. If this set is empty, we set $h_o:=\inf f$. Since $\Sigma_h$ is outer-minimizing in our setting, the function $\mathcal{V}(h)$ is lower semicontinuous and non-decreasing. 
Before the volume estimate, we establish the following lemma for the rate of volume growth of level sets. 
\begin{lemma}
Let $(\Omega,g)$ be a compact (locally) hyperbolic graphical manifold with scalar curvature $R(g)\geq -n(n-1)$. If $\mathbf{m}:={m}_{BY}^S(\Sigma)$, then for any regular level set $\Sigma_h$ with $\mathcal{V}(h)>\mathcal{A}_{n-1}\left(2c^{-1}_\epsilon \mathbf{m}\right)^{\frac{1}{n_\epsilon}}$, the rate of volume growth is 
\begin{equation}\label{volumegrowth}
\mathcal{V}'(h)>\frac{4(n-1)\mathcal{A}_{n-1}\mathbf{m}}{3\sqrt{3}}\left(\frac{c_\epsilon}{2\mathbf{m}}\left(\frac{\mathcal{V}(h)}{\mathcal{A}_{n-1}}\right)^{n_\epsilon}-1\right)^{3/2}\,.
\end{equation} 
\end{lemma}
\begin{proof}
Let $\Sigma_h$ be a regular level set of 
$\Omega$ and set $V_\epsilon=V$. Let $U_\alpha:=\{x\in\mathbb{M}_\epsilon^n:V|D f(x)|\geq \alpha\}$ and $L_\alpha:=\{x\in\mathbb{M}_\epsilon^n:V|D f(x)|< \alpha\}$, then by Lemma \ref{lemma2.1}, the quasi-local energy has the following lower bound
\begin{equation}
    \begin{aligned}
        2(n-1)\mathcal{A}_{n-1}\mathbf{m}&\geq\int_{\Sigma_h\cap U_\alpha}V\frac{V^2|Df|^2}{1+V^2|Df|^2}\mathring{H}_{\Sigma}\,dA_h\\
        &\geq\frac{\alpha^2}{1+\alpha^2}\int_{\Sigma_h\cap U_\alpha}V\mathring{H}_{\Sigma_h}dA_h.
    \end{aligned}
\end{equation}
This implies that
\begin{equation}\label{equa1}
    \begin{aligned}
        -\int_{\Sigma_h\cap U_\alpha}V\mathring{H}_{\Sigma_h}dA_h\geq -2(1+\alpha^{-2})(n-1)\mathcal{A}_{n-1}\mathbf{m}\,.
    \end{aligned} 
\end{equation}
Then, we can estimate the rate of volume growth as follows
 \begin{equation}  \label{equa2}         
    \begin{aligned}
    \mathcal{V}'(h)&=\int_{\Sigma_h}\frac{\mathring{H}_{\Sigma_h}}{|Df|}dA_h\\
     &=\int_{\Sigma_h\cap U_\alpha}\frac{\mathring{H}_{\Sigma_h}}{|Df|}dA_h+\int_{\Sigma_h\cap L_\alpha}\frac{\mathring{H}_{\Sigma_h}}{|Df|}dA_h\\
     &>\frac{1}{\alpha}\int_{\Sigma_h\cap L_\alpha}V\mathring{H}_{\Sigma_h}dA_h\\
     &=\frac{1}{\alpha}\left(\int_{\Sigma_h}V\mathring{H}_{\Sigma_h}dA_h-\int_{\Sigma_h\cap U_\alpha}V_\kappa\mathring{H}_{\Sigma_h}dA_h\right)\,.
     \end{aligned}
 \end{equation}
Combining \eqref{equa1} and \eqref{equa2}, we obtain
 \begin{equation}\label{equa23}
     \begin{aligned}
      \mathcal{V}'(h)\geq\frac{1}{\alpha}\left(\int_{\Sigma_h }V\mathring{H}_{\Sigma_h}dA_h-2(1+\alpha^{-2})(n-1)\mathcal{A}_{n-1}\mathbf{m}\right)\,.  
     \end{aligned}
 \end{equation}
Since all level sets are outer-minimizing, repeating argument in the proof of \eqref{min2} for $\mathcal{V}(h)$, we have
\begin{equation}\label{min2}
    \begin{split}
        \int_{\Sigma_h }V\mathring{H}_{\Sigma_h}dA_h&\geq (n-1)\mathcal{A}_{n-1}c_{\epsilon} \left(\frac{\mathcal{V}(h)}{\mathcal{A}_{n-1}}\right)^{n_{\epsilon}}
    \end{split}
\end{equation}
Together with \eqref{equa23}, we get
\begin{equation}\label{equa234}
\begin{aligned}
\mathcal{V}'(h)\geq\frac{2(n-1)\mathcal{A}_{n-1}}{\alpha}\left(\frac{c_{\epsilon}}{2} \left(\frac{\mathcal{V}(h)}{\mathcal{A}_{n-1}}\right)^{n_\epsilon}-(1+\alpha^{-2})\mathbf{m}\right)\,,
     \end{aligned}
 \end{equation}Consider the right hand side as a function of $\alpha$ as follows
 \begin{equation}
 G(\alpha)=\frac{2(n-1)\mathcal{A}_{n-1}}{\alpha}\left(\gamma-(1+\alpha^{-2})\mathbf{m}\right)\,,
 \end{equation}where $\gamma=2^{-1}c_{\epsilon} \left(\frac{\mathcal{V}(h)}{\mathcal{A}_{n-1}}\right)^{n_\epsilon}$. For $\mathcal{V}(h)>\mathcal{A}_{n-1}\left(2c^{-1}_\epsilon \mathbf{m}\right)^{\frac{1}{n_\epsilon}}$, the maximum of this function is at 
 \begin{equation}
     \alpha=\sqrt{3}\left[\frac{1}{\mathbf{m}}\gamma-1\right]^{-\frac{1}{2}}\,.
 \end{equation} Substituting this in \eqref{equa234}, we obtain the volume growth inequality \eqref{volumegrowth}.
\end{proof}

Next we include our height estimate for $\max f-h_o$. Note that we divide the result into two cases in order to present the correct quasi-local version of \cite[Lemma 4.8]{CabreraPacheco19} for future references. Otherwise, the main result of this paper addresses the convergence for $\overline{U}\backslash U_o$ with $n_\epsilon=\frac{1}{2}$ for all of our cases if for $\epsilon=1$ we set $c_\epsilon=\sqrt{1+r_0^{2}}r_0^{\frac{n-3}{2}}$.
\begin{lemma}\label{heightbound}
 Let $(\Omega,g)$ be a compact (locally) hyperbolic graphical manifold with $R(g)\geq -n(n-1)$. If $\mathbf{m}=m_{BY}^S(\Sigma)>0$, then for locally hyperbolic graphical manifold with $n\geq 3$ and three-dimensional hyperbolic graphical manifolds, we have
\begin{equation}\label{pppp1}
\begin{aligned}
0\leq \max(f)-h_o\leq C_{n,\epsilon}|\Sigma|^{\tfrac{1}{4}}\,\mathbf{m}^{\tfrac{1}{2}}\,,
\end{aligned}
\end{equation}and for compact hyperbolic graphical manifold with $n\geq 3$ we have
    \begin{equation}\label{pppp2}
        \begin{aligned}
            0\leq \max(f)-h_o\leq C_{n,\epsilon}\mathbf{m}^{\tfrac{1}{n-2}}\left(|\log\mathbf{m}|+|\Sigma|\right)\,,
        \end{aligned}
    \end{equation}where constant $C_{n,\epsilon}$ depends on $\epsilon$, $n$, $c_\epsilon$, and $\mathcal{A}_{n-1}$. 
\end{lemma}
\begin{proof}Assume $h_o<\text{max}(f)$, otherwise, there is nothing to prove. The proof uses \cite[Lemma 3.9]{HuangLee}. For function $Y(h)$, consider the differential equation 
\begin{equation}\label{l0}
\begin{aligned}
Y'(h)&=\frac{4(n-1)\mathcal{A}_{n-1}\mathbf{m}}{3\sqrt{3}}\left(\frac{c_\epsilon}{2\mathbf{m}}\left(\frac{Y(h)}{\mathcal{A}_{n-1}}\right)^{n_\epsilon}-1\right)^{3/2}\,,\\ 
Y(h_o)&=2(1+\xi)^{\frac{1}{n_\epsilon}} \mathcal{A}_{n-1}\left(2c^{-1}_\epsilon \mathbf{m}\right)^{\frac{1}{n_\epsilon}}.
\end{aligned}
\end{equation}First, we define the following positive function
\begin{equation}\label{ph}
p(h)=\frac{c_\epsilon}{2\textbf{m}}\left(\frac{Y(h)}{\mathcal{A}_{n-1}}\right)^{n_\epsilon}-1\,.
\end{equation}Then we have 
\begin{equation}\label{l1}
    \begin{aligned}
        p'(h)=\frac{n_\epsilon c_\epsilon}{2\textbf{m}}\frac{Y'(h)}{\mathcal{A}_{n-1}}\left(\frac{Y(h)}{\mathcal{A}_{n-1}}\right)^{n_\epsilon-1}\,,
    \end{aligned}
\end{equation}
Combining \eqref{l0} and \eqref{l1}, we have 
\begin{equation}\label{ppp1p1}
    \begin{aligned}
        h-h_o=D_{n,\epsilon}\textbf{m}^{\frac{1-n_\epsilon}{n_\epsilon}}\int_{h_o}^h\left(p+1\right)^{\frac{1-n_\epsilon}{n_\epsilon}}p^{-\frac{3}{2}}\,dp\,,
    \end{aligned}
\end{equation}where $D_{n,\epsilon}=\frac{3\sqrt{3}2^{\frac{1-n_\epsilon}{n_\epsilon}}}{2(n-1)n_\epsilon c_\epsilon^{\frac{1}{n_\epsilon}}}$. For locally hyperbolic graphical manifold $n\geq 3$ or three-dimensional hyperbolic graphical manifold, the constant $n_\epsilon$ is $n_\epsilon=\frac{1}{2}$. Therefore,  \eqref{ppp1p1} implies
\begin{equation}\label{eqrq1}
    h-h_o=2D_{n,\epsilon}\,\textbf{m}\left[\sqrt{p(h)}-\frac{1}{\sqrt{p(h)}}\right]_{h_o}^h\leq 4c_\epsilon\frac{D_{n,\epsilon}}{(4\mathcal{A}_{n-1})^{1/4}}\sqrt{\textbf{m}}\left[Y(h)\right]^{\frac{1}{4}}\,,
\end{equation}where for the last inequality we used $p(h_o)=\xi>1$, definition of $p(h)$, and $p(h)$ is a non-decreasing function for $h\geq h_o$. Observe that $\mathcal{V}(h)$ is non-decreasing function of $h$, $\mathcal{V}(h_o)\geq Y(h_o)$,  $\mathcal{V}'(h)\geq F(\mathcal{V}(h))$, and $Y'(h)=F(Y)$ for $h\geq h_o$, where $F$ is a non-decreasing $C^1$ function. Therefore, functions $\mathcal{V}(h)$ and $Y(h)$ have conditions of the comparison Lemma \cite[Lemma 3.9]{HuangLee} for solutions of first order ODE which implies $\mathcal{V}\geq Y(h)$ for $h\geq h_o$. Combining this with equation \eqref{eqrq1}, we obtain equation \eqref{pppp1}.

For compact hyperbolic graphical manifold with $n\geq 4$, we have
 \begin{equation}\label{hn}
     \begin{aligned}
         h-h_o=D_{n,\epsilon}\textbf{m}^{\frac{1-n_\epsilon}{n_\epsilon}}\int_{h_o}^h\left(p+1\right)^{\frac{1-n_\epsilon}{n_\epsilon}}p^{-\frac{3}{2}}\,dp\leq D_{n,\epsilon}\mathbf{m}^{\frac{1}{n-2}}\left(k(h)-k(h_o)\right)\,,
     \end{aligned}
 \end{equation}where 
 \begin{equation}
     k(h)=-2\frac{\sqrt{p(h)+1}}{\sqrt{p(h)}}+\log\left(\frac{1}{2}+p(h)+\sqrt{p(h)^2+p(h)}\right)\,,
 \end{equation}As shown in \cite{AlaeepachecoMccormick}, we have 
 \begin{equation}
     \begin{aligned}
        k(h)-k(h_o)\leq 6\log(1+p(h))\,.
     \end{aligned}
 \end{equation} Combine this with \eqref{ph} and \eqref{hn} lead to the result.
\end{proof}

\begin{lemma}\label{volume estimates}
      Let $(\Omega,g)$ be a compact (locally) hyperbolic graphical manifold with $R\geq -n(n-1)$. If $\mathbf{m}=m_{BY}^S(\Sigma)>0$, then for locally hyperbolic graphical manifold with $n\geq 3$ and three-dimensional hyperbolic graphical manifolds, we have
\begin{equation}
    \begin{aligned}
        \operatorname{vol}(\overline{U}\backslash U_o)\leq\operatorname{vol}(\Omega)\leq  \operatorname{vol}(\overline{U}\backslash U_o)+C\left(\mathbf{m}^{\tfrac{1}{2}}+\mathbf{m}^{\frac{2n}{n-1}}+V_{max}\left(h_o-\min f\right)\mathbf{m}^{2}\right)\,,
    \end{aligned}
\end{equation}and for hyperbolic graphical manifolds with $n\geq 4$, we have
\begin{equation}
    \begin{aligned}
        \operatorname{vol}(\overline{U}\backslash U_o)\leq\operatorname{vol}(\Omega)\leq &  \operatorname{vol}(\overline{U}\backslash U_o)\\
        &+C\left(\mathbf{m}^{\tfrac{1}{n-2}}\left(|\log\mathbf{m}|+|\Sigma|\right)+\mathbf{m}^{\frac{2n}{n-1}}+V_{max}\left(h_o-\min f\right)\mathbf{m}^{\frac{1}{n_\epsilon}}\right)\,,
    \end{aligned}
\end{equation}
where $V_{max}=\max_{h\leq h_o}V$, $\operatorname{vol}(\cdot)$ denotes the $n$-dimensional volume, and $C$ is a constant depend on $C_{iso}, n, |\Sigma|, C_{n,\epsilon}$ and $\xi$.
\end{lemma}
\begin{proof}Consider the compact graph $(\Omega,g)$. We define the region below $h_o$ in $\Omega$ by $\Omega_{h_o}=\{(x,f(x))\in\Omega: f(x)\leq h_o\}$ and $\Omega_{h_o}^c=\Omega\setminus\Omega_{h_o}$. Then 
\begin{equation}\label{p1}
   \operatorname{vol}(\overline{U}\backslash U_o)\leq  \operatorname{vol}(\Omega)\leq \operatorname{vol}(\Omega_{h_o})+\operatorname{vol}(\Omega^{c}_{h_o})\,.
\end{equation}Consider a sequence $\{h_i^-\}$ of regular values such that $h^-_i\leq h_o$ and $h_i^- \uparrow h_o$. Then using coarea formula, we have 
\begin{equation}
\begin{split}
\operatorname{vol}(\Omega_{h_o})&=\lim_{h_i^-\to h_o} \operatorname{vol}(\Omega_{h_i^-})=\lim_{h_i^-\to h_o}\int_{\Omega_{h_i^-}}dV_{g}=\lim_{h_i^-\to h_o}\int_{\mathcal{U}_{h_i^-}}\sqrt{1+V^2|Df|^2}dV_{b_\epsilon}\\
&\leq \lim_{h_i^-\to h_o}\int_{\mathcal{U}_{h_i^-}}\left(1+V|Df|\right)dV_{b_\epsilon}\leq \lim_{h_i^-\to h_o}\left(\operatorname{vol}(\mathcal{U}_{h_i^-})+V_{max}\int_{\mathcal{U}_{h_i^-}}|Df|dV_{b_\epsilon}\right)\\
&=\lim_{h_i^-\to h_o}\left(\operatorname{vol}(\mathcal{U}_{h_i^-})+V_{max}\int_{\text{min}f}^{h_i^-}\mathcal{V}(h)\,dh\right)\,,
\end{split}
\end{equation}where $\mathcal{U}_{h_i^-}=\{x\in \overline{U}\backslash U_o: f(x)<h_i^-\}$ and $V_{\max}=\max_{h\leq h_o}V$. Recall the isoperimetric inequality for compact domain $M$ with boundary $\partial M$ is \cite[equation 3.1.4.] {SchoenYau94}
\begin{equation}\label{iso}
\operatorname{vol}(M)\leq C_{iso} \operatorname{vol}(\partial M)^{\frac{n}{n-1}}\,,
\end{equation}where $C_{iso}$ depends on geometry of $M$ and $n$. Therefore, using the outer-minimizing condition of level sets, we have
\begin{equation}
\operatorname{vol}(\mathcal{U}_{h_i^-})\leq C_{iso}\left(|\Sigma_0|+\mathcal{V}(h_o)\right)^{\frac{n}{n-1}}\leq C_{iso}\left(2\mathcal{V}(h_o)\right)^{\frac{n}{n-1}}\,,
\end{equation}which leads to
\begin{equation}\label{p2}
\begin{split}
\operatorname{vol}(\Omega_{h_o})&\leq C_{iso}\left(2\mathcal{V}(h_o)\right)^{\frac{n}{n-1}}+V_{max}\left(h_o-\min f\right)\max_{h\leq h_0} \mathcal{V}(h)\\
&\leq 2 C_{iso}4^{\frac{n}{n-1}}(1+\xi)^{\frac{n}{n_\epsilon (n-1)}} \mathcal{A}_{n-1}^{\frac{n}{n-1}}\left(2c^{-1}_\epsilon \mathbf{m}\right)^{\frac{n}{n_\epsilon (n-1)}}\\
&+V_{max}\left(h_o-\min f\right)2(1+\xi)^{\frac{1}{n_\epsilon}} \mathcal{A}_{n-1}\left(2c^{-1}_\epsilon \mathbf{m}\right)^{\frac{1}{n_\epsilon}}\,.
\end{split}
\end{equation}Similarly, consider a sequence of regular values $h_i^+\downarrow h_o$ and following similar steps as above we obtain
\begin{equation}\label{p3}
\begin{split}
\operatorname{vol}(\Omega_{h_o}^c)&\leq \operatorname{vol}(\overline{U}\backslash U_o)+|\Sigma|\left(\max f-h_o\right)\,.
\end{split}
\end{equation}
Combing equations \eqref{p1}, \eqref{p2}, \eqref{p3} with Lemma \ref{heightbound} we achieve the estimate.
\end{proof}

\section{Proof of the Theorem \ref{main result2}}
In this section, we will present the main result of this paper. Given a Riemannian manifold $M$, it is essential to recall that a submanifold $N$ can be viewed as an integral current $T$ of multiplicity one. The boundary $\partial N$ is also an integral current in this case. Furthermore, the mass of a submanifold $N$ is denoted by $\textbf{M}(N)$ which is simply the volume of the submanifold $N$.

\begin{definition}\label{flat distance}
    Let $U_o\subset U$ be an open sets in $\mathbb{M}^{n}_\epsilon$, and let $T_1$ and $T_2$ be integral k-currents in $Z= (a,b)\times (\overline{U}\backslash U_o)$, and let $\bold{M}_Z$ denote the mass of a current in $Z$, the flat distance between $T_1$ and $T_2$ in Z is defined as 
    \begin{equation}
        \begin{aligned}
            d_F^Z(T_1, T_2)= \inf\{\bold{M}_Z(A) + \bold{M}_Z(B) : T_1\,\backslash\,T_2 = A +\partial B\},
        \end{aligned}
    \end{equation}
    where the infimum is taken over all integral $k$-currents $A$, and all integral currents $B$ in $Z=(a,b)\times (\overline{U}\backslash U_o)$, such that $T_1\,\backslash\,T_2=A+\partial B$.
\end{definition}

\begin{theorem}\label{main estimates}
    Let $(\Omega,g)$ be a compact (locally) hyperbolic graphical manifold with scalar curvature $R(g)\geq -n(n-1)$. Then 
    \begin{equation}
        \begin{aligned}
             d_F^{\Rr\times (\overline{U}\backslash U_o)}(\Omega, \{h_o\}\times (\overline{U}\backslash U_o))\leq C\mathbf{m}^\gamma\,,
        \end{aligned}
    \end{equation}for some $\gamma>0$ and $C$ depends on $C_{n,\epsilon}$, $\mathcal{A}_{n-1}$, $n_\epsilon$, $c_\epsilon$, $|\Sigma|$, $\operatorname{vol}\left(\overline{U}\backslash U_o\right)$, and $C_{iso}$.
\end{theorem}
\begin{proof}To estimate the distance between $\Omega$ and $\{h_o\}\times (\overline{U}\backslash U_o)$, we divide the region between these two submanifolds in $\mathbb{M}^{n+1}_{\epsilon}$ into the following four sets
    \begin{equation}
        \begin{aligned}
            B^+:&=\{(x,s)\in (\overline{U}\backslash {U}_o)\times \mathbb{R}|\, h_o\leq s\leq \max f \},\\
            B^-:&=\{(x,s)\in (\overline{U}\backslash {U}_o)\times \mathbb{R}|\, \min f\leq s\leq h_o \},\\
            A_+:&=\{(x,s)\in \partial U\times \mathbb{R}|\, h_o\leq s\leq \max f \},\\
            A_-:&=\{(x,s)\in \partial U_o\times \mathbb{R}|\, \min f\leq s\leq h_o \}\,.
        \end{aligned}
    \end{equation}
Therefore, $\Omega\,\backslash\,\left(\{h_o\}\times (\overline{U}\backslash U_o)\right)=\partial(B_++B_-)+A_-+A_+$. Note that $A_+\cap A_-=\emptyset$ and $B_+\cap B_-=\emptyset$, so we have $\mathbf{M}(B_+\cup B_-)=\mathbf{M}(B_+)+\mathbf{M}(B_-)$ and $\mathbf{M}(A_+\cup A_-)=\mathbf{M}(A_+)+\mathbf{M}(A_-)$. For any regular value $h\leq h_o$, by definition of $h_o$ we have
\begin{equation}
    \begin{aligned}
        \mathcal{V}(h)\leq 2(1+\xi)^{\frac{1}{n_\epsilon}} \mathcal{A}_{n-1}\left(2c^{-1}_\epsilon \mathbf{m}\right)^{\frac{1}{n_\epsilon}}.
    \end{aligned}
\end{equation}
By the isoperimetric inequality \eqref{iso}, we have 
\begin{equation}
    \begin{aligned}
        \mathbf{M}(B_-)&\leq \int_{\min f}^{h_o}\operatorname{vol}(\mathcal{U}_h)\, dh\\
        &\leq C_{iso}\int_{\min f}^{h_o}\left(2\mathcal{V}(h)\right)^{\frac{n}{n-1}}\, dh\\
        &\leq 2C_{iso}(h_o-\min f)4^{\frac{n}{n-1}}(1+\xi)^{\frac{n}{n_\epsilon (n-1)}} \mathcal{A}_{n-1}^{\frac{n}{n-1}}\left(2c^{-1}_\epsilon \mathbf{m}\right)^{\frac{n}{n_\epsilon (n-1)}},
    \end{aligned}
\end{equation}where the second one follows from that fact that $|\partial U_o|=|\Sigma_o|\leq \mathcal{V}(h)\leq\mathcal{V}(h_o)$. Similarly, we have  
\begin{equation}
    \begin{aligned}
        \mathbf{M}(B_+)&\leq \left(\max f- h_o\right)\operatorname{vol}\left(\overline{U}\backslash U_o\right)\,.
    \end{aligned}
\end{equation}Combining this with Lemma
\ref{heightbound}, then for compact locally hyperbolic graphical manifold  with $n\geq 3$ and three-dimensional compact hyperbolic graphical manifold 
\begin{equation}
    \begin{aligned}
        \mathbf{M}(B_+)&\leq C_{n,\epsilon}\operatorname{vol}\left(\overline{U}\backslash U_o\right)|\Sigma|^{\frac{1}{4}}\mathbf{m}^\frac{1}{2}\,,
    \end{aligned}
\end{equation} and for
compact hyperbolic graphical manifold with $n\geq 4$ we have
\begin{equation}
    \begin{aligned}
        \mathbf{M}(B_+)&\leq C_{n,\epsilon}\operatorname{vol}\left(\overline{U}\backslash U_o\right)\mathbf{m}^\frac{1}{n-1}\left(|\log\mathbf{m}|+|\Sigma|\right)\,.
    \end{aligned}
\end{equation}
Next, we estimate mass of $A_+$ and $A_-$.
\begin{equation}
\begin{split}
\mathbf{M}(A_-)&\leq (h_o-\min f)|\Sigma_o|\\
&\leq (h_o-\min f)\mathcal{A}_{n-1}\left(2c_{\epsilon}^{-1}\mathbf{m}\right)^{\frac{1}{n_\epsilon}}\,,
\end{split}
\end{equation}where we used the Riemannian Penrose inequality, Lemma \ref{lemma2.2}. Similarly, since every level set is outer-minimizing we get 
\begin{equation}
\mathbf{M}(A_+)\leq (\max f-h_o)|\Sigma|\,.
\end{equation}
Together with Lemma \ref{heightbound}, for compact locally hyperbolic graphical manifold  with $n\geq 3$ and three-dimensional compact hyperbolic graphical manifold we obtain
\begin{equation}
\begin{split}
\mathbf{M}(A_+)&\leq (\max f-h_o)|\Sigma|\leq  C_{n,\epsilon}|\Sigma|^{\frac{5}{4}}\mathbf{m}^{\frac{1}{2}}\,,
\end{split}
\end{equation}
and for
compact hyperbolic graphical manifold with $n\geq 4$ we obtain
\begin{equation}
\begin{split}
\mathbf{M}(A_+)&\leq C_{n,\epsilon}|\Sigma|\mathbf{m}^{\tfrac{1}{n-2}}\left(|\log\mathbf{m}|+|\Sigma|\right)\,,
\end{split}
\end{equation}
The result follows from these estimates and  
\begin{equation}
        \begin{aligned}
             d_F^{\Rr\times (\overline{U}\backslash U_o)}(\Omega, \{h_o\}\times (\overline{U}\backslash U_o))\leq \mathbf{M}(A_+)+\mathbf{M}(A_-)+\mathbf{M}(B_-)+\mathbf{M}(B_+)\,.
        \end{aligned}
    \end{equation}
\end{proof}

\begin{proof}[Proof of Theorem \ref{main result2}] After rescaling the graph such that $h_o$ is $t=0$ hypersurface in $\mathbb{M}_{\epsilon}^{n+1}$, the proof follows directly from Lemma \ref{volume estimates} and Theorem \ref{main estimates}.   
\end{proof}

\end{document}